\pdfoutput=1
\RequirePackage{ifpdf}
\ifpdf % We~are running pdfTeX in pdf mode
\documentclass[pdftex]{sigma}%,draft
\else
\documentclass{sigma}
\fi

\numberwithin{equation}{section}

\newtheorem{Theorem}{Theorem}[section]
\newtheorem*{Theorem*}{Theorem}

\newtheorem{Proposition}[Theorem]{Proposition}
 { \theoremstyle{definition}
\newtheorem{Definition}[Theorem]{Definition}
\newtheorem{Note}[Theorem]{Note}
\newtheorem{Example}[Theorem]{Example}
\newtheorem{Remark}[Theorem]{Remark} }

\begin{document}
\allowdisplaybreaks

\renewcommand{\thefootnote}{}

\newcommand{\arXivNumber}{2108.08082}

\renewcommand{\PaperNumber}{028}

\FirstPageHeading

\ShortArticleName{Pullback Coherent States, Squeezed States and Quantization}

\ArticleName{Pullback Coherent States, Squeezed States\\ and Quantization\footnote{This paper is a~contribution to the Special Issue on Mathematics of Integrable Systems: Classical and Quantum in honor of Leon Takhtajan.

~~\,The full collection is available at \href{https://www.emis.de/journals/SIGMA/Takhtajan.html}{https://www.emis.de/journals/SIGMA/Takhtajan.html}}}

\Author{Rukmini DEY and Kohinoor GHOSH}

\AuthorNameForHeading{R.~Dey and K.~Ghosh}

\Address{International Center for Theoretical Sciences, Sivakote, Bangalore, 560089, India}
\Email{\href{mailto:rukmini@icts.res.in}{rukmini@icts.res.in}, \href{mailto:kohinoor.ghosh@icts.res.in}{kohinoor.ghosh@icts.res.in}}
\URLaddress{\url{https://www.icts.res.in/people/rukmini-dey},\newline
\hspace*{10.5mm}\url{https://www.icts.res.in/people/kohinoor-ghosh}}

\ArticleDates{Received December 07, 2021, in final form March 30, 2022; Published online April 09, 2022}

\Abstract{In this semi-expository paper, we define certain Rawnsley-type coherent and squeezed states on an integral K\"ahler manifold (after possibly removing a set of measure zero) and show that they satisfy some properties which are akin to maximal likelihood pro\-perty, reproducing kernel property, generalised resolution of identity property and overcompleteness. This is a generalization of a result by Spera. Next we define the Rawnsley-type pullback coherent and squeezed states on a smooth compact manifold (after possibly remo\-ving a set of measure zero) and show that they satisfy similar properties. Finally we show a Berezin-type quantization involving certain operators acting on a Hilbert space on a compact smooth totally real embedded submanifold of $U$ of real dimension $n$, where $U$ is an open set in ${\mathbb C}{\rm P}^n$. Any other submanifold for which the criterion of the identity theorem holds exhibit this type of Berezin quantization. Also this type of quantization holds for totally real submanifolds of real dimension $n$ of a general homogeneous K\"ahler manifold of real dimension $2n$ for which Berezin quantization exists. In the appendix we review the Rawnsley and generalized Perelomov coherent states on ${\mathbb C}{\rm P}^n$ (which is a coadjoint orbit) and the fact that these two types of coherent states coincide. }


\Classification{53D50; 53D55}

\renewcommand{\thefootnote}{\arabic{footnote}}
\setcounter{footnote}{0}

\section{Introduction}

Let $M$ be a compact K\"ahler manifold with $\omega$ an integral K\"ahler form and let $L$ be a prequantum line bundle (obtained from geometric quantization), i.e., its curvature is proportional to the K\"ahler form. Then one can take as the Hilbert space of
quantization the space of holomorphic sections of $ L^{\otimes \mu}$ for $\mu \in {\mathbb Z}$ large
enough. See \cite{W} for an explanation. Coherent states arise very naturally in geometric quantization. The Hilbert space obtained in geometric quantization could provide the starting point of Berezin quantization. For details on Berezin quantization see Berezin \cite{Be}, Perelomov \cite{Pe}.

The mathematical physics literature on coherent states is vast, see for example \cite{KS}. To name a few other works, we mention \cite{BS, Ki, Na, Od, Pe, Ra, Sp, Ya}.
For a survey on squeezed states, see for example \cite{Sch}.

In this article we will be focussing greatly on Rawnsley coherent states.
In \cite{Ra}, Rawnsley has defined coherent states on a compact K\"ahler manifold with an integral
K\"ahler form which arise naturally out of geometric quantization.
This goes as follows.
Let~${\mathcal H}$ be space of holomorphic square integrable sections of the quantum bundle $L$, the mea\-sure being~${\rm e}^{-F} {\rm d}\mu$, where $F$ is the K\"ahler potential and ${\rm d}\mu $ proportional to the volume form.
Rawnsley defines the coherent states by considering the section-evaluation functional and writing it as an inner product with a coherent state vector (using Riesz representation theorem).

Spera \cite{Sp} had shown that under certain conditions the Rawnsley coherent states satisfy the maximal likelihood property, reproducing kernel property, the generalised resolution of identity and overcompleteness.

We define certain Rawnsley-type coherent states and squeezed states on an integral compact K\"ahler manifold.
We generalize the result by Spera to show that this type of coherent states sati\-sfy properties akin to maximal likelihood property, reproducing kernel property, the generalised resolution of identity and overcompleteness.

The high point of this paper is the definition of certain Rawnsley-type pullback coherent states on an arbitrary compact smooth manifold~$M$. $M$~need not have a symplectic structure and hence no geometric quantization. It may not have a group action as well. But we can still talk of Rawnsley-type coherent states on it, using pullback by a~smooth embedding in ${\mathbb C}{\rm P}^n$. The Rawnsley-type coherent states (which are related to the pull backs of the coherent states on ${\mathbb C}{\rm P}^n$) are given. We show that these satisfy properties akin to the maximal likelihood property, reproducing kernel property, the generalised resolution of identity and overcompleteness. Similarly we define the pullback squeezed states and show that they exhibit similar properties. These definitions depend on the embedding. The pullback coherent states are useful in defining Berezin quantization of arbitrary smooth submanifolds of~${\mathbb C}{\rm P}^n$ under certain conditions. This is work in progress.

Finally we show a Berezin-type quantization involving certain operators acting on a Hilbert space on compact smooth totally real submanifolds of ${\mathcal U}$ of real dimension $n$, where ${\mathcal U}$ is an open set of ${\mathbb C}{\rm P}^n$ which is biholomorphic to ${\mathbb C}^n $. The conditions ensure that the identity theorem holds, namely if two holomorphic functions defined on ${\mathcal U}$ agree on the submanifold, they must agree on~${\mathcal U}$. Any other submanifold for which the criterion of the identity theorem holds (see~\cite[Chapter~9, Lemma~2]{Bo}) will also exhibit this type of Berezin quantization.

There is nothing special about ${\mathbb C}{\rm P}^n$ as this type of quantization holds for totally real submanifolds of real dimension $n$ of a general homogenous K\"ahler manifold of real dimension $2n$ (or any other submanifolds for which the identity theorem criterion holds).

Perelomov coherent states are defined by Perelomov as generalized coherent states in \cite[p.~40]{Pe}.
The generalized coherent states of geometric quantization of ${\mathbb C}{\rm P}^1$ have been described in Radcliffe \cite{Rad} and Nair \cite[p.~501]{Na} in the context of spin quantization. Recall that ${\mathbb C}{\rm P}^1$ is a~coadjoint orbit of the form $\frac{{\rm SU}(2)}{{\rm U}(1)}$ and the generalized coherent states can be found in the sense of Perelomov \cite{Pe} and Rawnsley \cite{Ra}.
The connection between Borel--Weil theory and Perelomov coherent states is reviewed in Spera~\cite{Sp2}.

In the appendix we review the generalised Perelomov coherent states (defined in \cite[p.~40]{Pe}) for ${\mathbb C}{\rm P}^n$ which is a coadjoint orbit of the form $\frac{{\rm SU}(n+1)}{{\rm S}({\rm U}(n) \times {\rm U}(1))}$. We use the general result in~\cite{K}. We then explicitly describe the Rawnsley coherent states and review the fact that they are the same as the Perelomov coherent states for ${\mathbb C}{\rm P}^n$ following an argument in~\cite{Ra}.

\section[Rawnsley-type coherent and squeezed states on an integral compact K\"ahler manifold]
{Rawnsley-type coherent and squeezed states \\on an integral compact K\"ahler manifold}

\subsection{Rawnsley-type coherent states on integral compact K\"ahler manifold}

In Spera \cite{Sp} (amongst other papers and books), the properties of Rawnsley coherent states, namely, maximum likelihood property, reproducing kernel property, generalized resolution of identity,
overcompleteness etc. have been spelt out under certain conditions. We define certain Rawnsley-type coherent states on an integral compact K\"ahler manifold. We do this by modifying some of the ideas in the articles by Rawnsley \cite{Ra}, Kirwin \cite{Ki}, Spera \cite{Sp}, Perelomov \cite{Pe}.
We show that these Rawnsley-type coherent states satisfy properties akin to the properties mentioned above, thus generalizing Spera's result, i.e., we do not impose any conditions except that the K\"ahler form is integral (and thus there is a large enough $k$ such that $L^k$ is very ample where $L$ is a prequantum bundle). They are called Rawnsley-type because we show that they arise from a modification of the section-evaluation functional.

Let $(L,\nabla, 2\pi {\rm i} \Omega)\longrightarrow (M,\Omega)$ be a line bundle with $\Omega$ an integral K\"ahler form on $M$. Let $h$ be a Hermitian metric on $L^k$, where $k$ is such that $L^k$ is very ample. Let $\Gamma$ be the space of global holomorphic sections of~$L^k$.

We choose the inner product on $\Gamma$ w.r.t.\ $h$, namely $\langle\phi_1, \phi_2\rangle = \int_{M} \overline{\phi_1} \phi_2 h {\rm d}V $, to be antilinear in the first term and linear in the second (unlike Rawnsley's convention). We will retain this convention henceforth. Let ${\mathcal H}$ be the space of square integrable sections in $\Gamma$.
Let $\{\psi_i\}_{i=1}^m$ be an orthonormal basis for ${\mathcal H}$ which is basepoint free.
Let $\phi \in \mathcal{H}$, a square integrable holomorphic section of $L^k$ then $\phi$ can be expressed as a linear combination of the orthonormal basis ele\-ments~$\{\psi_i\}_{i=1}^m$. Then $\phi = \sum_{i=1}^m\langle\psi_i, \phi\rangle\psi_i $.
Let $\mu \in M $ and $\chi(\mu)^2 = \sum_{i=1}^m |\psi_i(\mu)|^2 $. Since $\{\psi_i\}_{i=1}^m$ is a~basepoint free basis, they do not simultaneously vanish. Thus $\chi(\mu) \neq 0$.
Let $\psi_i = f_i s_0$, where~$s_0$ is a fixed section such that its zero set is $M_0 \subset M$ and $f_i$ are meromorphic functions which are holomorphic on $M \setminus M_0$.

\begin{Definition}
For $\mu \in M$ we define
\[
\phi_{\mu}= \frac{1}{p(\mu)} \sum_{i=1}^m \overline{f_i(\mu)} \psi_i,
\]
where $\psi_i = f_i s_0$, $f_i$ is a meromorphic function on $M$, which is holomorphic on $M \setminus M_0 $ and $p(\mu)^2 = \sum_{i=1}^m |f_i(\mu)|^2$.
\end{Definition}

For $\mu \in M \setminus M_0$ one sees that
\[
\phi_{\mu}= \frac{s_0(\mu)}{|s_0(\mu)|\chi (\mu)} \sum_{i=1}^m \overline{\psi_i (\mu)} \psi_i,
\]
where $\chi(\mu)^2 = \sum_{i=1}^m |\psi_i(\mu)|^2 \neq 0$.
 Note that $\phi_{\mu}$ is a smooth section of $L^k$.
One can check that $\| \phi_{\mu}\|^2 =1$ and $\phi_{\mu} (\mu) =\tau(\mu)$, where $\tau(\mu) = s_0 p(\mu) = \frac{s_0(\mu)}{|s_0(\mu)|}\chi(\mu)$. It is non zero for $ \mu \in M \setminus M_0$ since $\chi$ is non zero and $s_0(\mu) \neq 0$.

We have a generalization of Spera's result \cite[Theorem~2.1]{Sp} as follows:

\begin{Theorem}\label{Kah}
Let $M$ be a integral compact K\"ahler manifold and let $\mathcal{H}$, $s_0$, $\phi_{\mu}$ and $\tau$ be as defined above. Then, for all $\mu \in M$, we have
\begin{enumerate}\itemsep=0pt

\item[$(a)$] $\phi_{\mu}$ are Rawnsley-type coherent states of $M$.
In fact, for $\mu \in M$ the general formula is $\langle\phi_{\mu}, \phi\rangle = \frac{g(\mu)}{ p(\mu)} $, where $\phi = g s_0$.
 Also, $\langle\phi_{\mu}, \phi\rangle = \frac{\phi(\mu)}{ \tau(\mu)} $ for $\mu \in M \setminus M_0$.

\item[$(b)$] $\phi_{\mu} $ satisfy
\begin{enumerate}\itemsep=0pt

\item[$(i)$] the maximal likelihood property, namely, $|\phi_{\mu}(\mu)|^2 \geq | \phi (\mu)|^2 $ for all $\mu$ and all $\phi \in {\mathcal H} $ such that $\langle\phi, \phi\rangle =1$,

\item[$(ii)$] the property that $ | \phi_{\mu} (\mu)|^2 \geq |\phi_{\mu^{\prime}} (\mu)|^2$,

\item[$(iii)$] the modified reproducing kernel property, namely,
\[
\langle\phi_{\mu}, \phi\rangle = \frac{1}{\chi(\mu)^2} \overline{\phi_{\mu}(\mu)} \phi(\mu)\qquad
\text{for}\quad \mu \in M,
\]

\item[$(iv)$] the modified generalized resolution of identity, namely,
\[
\langle\psi_1, \psi_2\rangle = \int_{M} \langle\psi_1, \phi_{\mu}\rangle \langle\phi_{\mu}, \psi_2\rangle \chi(\mu)^2 h(\mu)\,{\rm d}V(\mu),
\]

\item[$(v)$] overcompleteness, namely, $\langle\phi_{\mu}, \psi\rangle =0$ $ \forall \mu$ imples $\psi=0$.
\end{enumerate}
\end{enumerate}
\end{Theorem}

\begin{proof}
$(a)$ follows from a simple calculation.
$(b)$ The proof follows by modifications of~\cite{Sp} as below:

$(i)$ $|\phi_{\mu}|^2 = \chi^2 $ and $\|\phi_{\mu}\|^2 =1$ and by $(a)$ we have
\[
|\phi(\mu)|^2 = | \langle\phi_{\mu}, \phi\rangle|^2 \chi(\mu)^2 \leq \|\phi_{\mu}\|^2 \|\phi\|^2 \chi(\mu)^2 = \chi(\mu)^2 = |\phi_{\mu}(\mu)|^2.
\]

$(ii)$ This follows along similar lines as $(i)$.

$(iii)$ This follows from $(a)$ and the fact that $\phi_{\mu}(\mu) = s_0 p(\mu)$ $(=\tau(\mu))$.

$(iv)$ Using $(a)$,
\begin{align*}
\langle\psi_1, \psi_2\rangle &= \int_M \overline{\psi}_1(\mu) \psi_2(\mu) h (\mu)\,{\rm d}V (\mu)
 = \int_{M} \overline{\langle\phi_ {\mu}, \psi_1\rangle} \langle\phi_{\mu}, \psi_2\rangle \chi(\mu)^2 h(\mu)\,{\rm d}V(\mu)
\\
&= \int_{M } \langle\psi_1, \phi_{\mu}\rangle \langle\phi_{\mu}, \psi_2\rangle \chi (\mu)^2 h(\mu)\,{\rm d}V(\mu).
\end{align*}

$(v)$ This follows from $(a)$.
\end{proof}

\subsection{Squeezed states on an integral compact K\"ahler manifold }

We define squeezed states in a similar fashion as the coherent states in the previous section.
 Let~$M$ be an integral compact K\"ahler manifold of dimension $d$.
Let $q \in U \subset M$, where $U$ is an open neighbourhood of $q$ such that $\phi_U $ is a biholomorphism to an open ball $V \subset {\mathbb C}^d$ to $U$.
Let $q=\nu= \phi_U( \nu_1 + {\rm i}\nu_2) $, where $\nu_1 + {\rm i}\nu_2 \in V$. Let $\zeta \in {\mathbb R}$ be such that $ \nu_1 + {\rm i}\zeta \nu_2 $ belongs to $V$. Then $\nu_{\zeta} = \phi_U( \nu_1 + {\rm i}\zeta \nu_2) \in U$. Let $q_{\zeta} = \nu_{\zeta}$.

\begin{Proposition}\label{celldecomp}
Let $M$ be a $2d$-dimensional compact smooth manifold. Then there exists a~subset
$\tilde{M}$ of dimension at most $(2d-1)$such that $M \setminus \tilde{M}$ admits an open cover by a single open set~$U$ of the above kind.
\end{Proposition}

\begin{proof}
A $2d$-dimensional manifold admits a cell-decomposition as a single $2d$-dimensional cell glued
to a skeleton of dimension at most $(2d-1)$ \cite{DH}. Let $\tilde{M}$ be this skeleton. Then $M \setminus \tilde{M}$ is homeomorphic to
$ {\mathbb C}^d$. Thus one open set $U$ is enough to cover $M \setminus \tilde{M}$.
\end{proof}

Let $\{\psi_i\}_{i=1}^m$ be the orthonormal basis for the Hilbert space of geometric quantization as described in the previous section, with inner product $\langle\cdot, \cdot\rangle$.
Let $\mu = \phi_U (\mu_1 + {\rm i}\mu_2)$ and $\mu_{\zeta} = \phi_U(\mu_1 + {\rm i}\zeta \mu_2)$.
We define
$\chi(\mu_{\zeta}) ^2 = \sum_{i=1}^m | \psi_i(\mu_{\zeta})|^2 $ which is non-zero since $\{\psi_i\}$ is a base-point free basis.

Let us define the squeezed states as follows.
Let $s_0$ be a fixed holomorphic section of the prequantum bundle whose vanishing set is $M_0 \subset M$.
Then for $\mu_{\zeta} \in M \setminus \big(\tilde{M} \cup M_0\big)$ we define the squeezed states as follows.
\begin{Definition}
\[
\phi^{\zeta}_{\mu} (\nu) = \frac{s_0(\mu_{\zeta}) }{| s_0(\mu_{\zeta})|\chi(\mu_{\zeta})} \sum_{i=1}^m \overline{\psi_i(\mu_{\zeta})} \psi_i(\nu).
\]
\end{Definition}

Let $\tau(\mu_{\zeta} ) = \frac{s_0(\mu_{\zeta})\chi(\mu_{\zeta}) }{| s_0(\mu_{\zeta})|}$.
It can be shown that they satisfy properties same as mentioned in the theorem below.

\begin{Theorem}\label{Sq2}

The $\phi^{\zeta}_{\mu}$ satisfy the following properties:
\begin{enumerate}\itemsep=0pt

\item[$1.$] $\big\langle\phi^{\zeta}_{\mu}, \phi^{\zeta}_{\mu}\big\rangle =1$.

\item[$2.$]
$\big\langle\phi^{\zeta}_{\mu}, \phi\big\rangle = \frac{1}{\tau(\mu_{\zeta})} \phi (\mu_{\zeta})$.

\item[$3.$] The maximal likelihood property, namely, $\big|\phi^{\zeta}_{\mu}(\mu_{\zeta})\big|^2 \geq | \phi (\mu_{\zeta})|^2 $ for all $\mu_{\zeta} \in U$ and all $\phi \in {\mathcal H} $ such that $\langle\phi, \phi\rangle =1$.

\item[$4.$] The property that $ \big| \phi^{\zeta}_{\mu} (\mu_{\zeta})\big|^2 \geq \big|\phi^{\zeta} _{\mu^{\prime}} (\mu_{\zeta})\big|^2$.

\item[$5.$] The modified reproducing kernel property, namely,
$\big\langle\phi^{\zeta}_{\mu}, \phi\big\rangle = \frac{1}{\chi(\mu_{\zeta})^2} \overline{\phi^{\zeta}_{\mu}(\mu_{\zeta})} \phi(\mu_{\zeta}) $ for \mbox{$\mu_{\zeta} \!\in\! U $}.

\item[$6.$] Modified resolution of identity:
$\langle\phi_1, \phi_2\rangle = \int_M \chi (\mu_{\zeta})^2 \big\langle\phi_1, \phi_{\mu}^{\zeta}\big\rangle
\big\langle\phi_{\mu}^{\zeta}, \phi_2\big\rangle h(\mu_{\zeta})\, {\rm d}V(\mu_{\zeta})$.

\item[$7.$] Overcompleteness: $\big\langle\phi_{\mu}^{\zeta}, \phi\big\rangle=0$ for all $\mu$ iff $\phi =0$.
\end{enumerate}
\end{Theorem}

\begin{proof}The proof is simple and similar to Theorem \ref{Kah}. We have to use crucially that $\tilde{M} \cup M_0$ is of measure zero and removing this set does not affect the integrals.
\end{proof}

For $\mu_{\zeta} \in M \setminus \big(\tilde{M} \cup M_0\big)$ one can define the second type of squeezed states by
\begin{Definition}
\[
\tilde{\phi}^{\zeta}_{\mu} (\nu) = \frac{s_0(\mu_{\zeta})}{|s_0(\mu_{\zeta})| \chi(\mu_{\zeta})} \sum_{i=1}^m \overline{\psi_i(\mu)} \psi_i(\nu_{\zeta}).
\]
\end{Definition}

Let $\tilde{\psi}_i (\nu) = \psi_i(\nu_{\zeta})$. Then $\tilde{\psi}$ is again a section and has an expansion in terms of the basis $\{\psi_{k}\}_{k=1}^m $.
Let $b_{ik}^{\zeta}$ be complex numbers such that
\begin{equation}\label{b}
\psi_i(\nu_{\zeta}) = \sum_{k=1}^m b_{ik}^{\zeta} \psi_{k}(\nu).
\end{equation}
\begin{Proposition}
If $b_{ik}^{\zeta} = \overline{b_{ki}^{\zeta}}$, then $\phi^{\zeta}_{\mu} = \tilde{\phi}^{\zeta}_{\mu}$.
\end{Proposition}

\begin{proof}
\[
\phi^{\zeta}_{\mu}(\nu) = \frac{s_0(\mu_{\zeta})}{|s_0(\mu_{\zeta})| \chi(\mu_{\zeta})} \sum_{i,k=1}^m \overline{b^{\zeta}_{ik} \psi_k(\mu)} \psi_i(\nu) = \frac{s_0(\mu_{\zeta})}{|s_0(\mu_{\zeta})| \chi(\mu_{\zeta})} \sum_{i,k=1}^m \overline{b^{\zeta}_{ki} \psi_i(\mu)} \psi_k(\nu)
\]
by interchanging dummy indices, and
\[
\tilde{\phi}^{\zeta}_{\mu} (\nu) =\frac{s_0(\mu_{\zeta}) }{| s_0(\mu_{\zeta})| \chi(\mu_{\zeta})} \sum_{i=1}^m \overline{\psi_i(\mu)} \psi_i(\nu_{\zeta}) = \frac{s_0(\mu_{\zeta})}{|s_0(\mu_{\zeta})|\chi(\mu_{\zeta})} \sum_{i,k=1}^m \overline{\psi_i(\mu)} b^{\zeta}_{ik}\psi_k(\nu).
\]
Since $b_{ik}^{\zeta} = \overline{b_{ki}^{\zeta}}$ we have the result.
\end{proof}

The second type of squeezed states also have the properties akin to coherent states.

\begin{Theorem}
If the orthonormal basis satisfies the condition $\overline{\psi(\nu)} = \psi(\overline{\nu})$ then
 $\tilde{\phi}^{\zeta}_{\mu}$ satisfy the following properties $(1)$--$(7)$ as in Theorem~$\ref{Sq2}$.
\end{Theorem}

\begin{proof}
Let $b_{ik}^{\zeta}$ be complex numbers such that
$\psi_i(\nu_{\zeta}) = \sum_{k=1}^m b_{ik}^{\zeta} \psi_{k}(\nu)$ as in equation (\ref{b}).

Proof of $(1)$ goes as follows:
\[
\big\langle \tilde{\phi}^{\zeta}_{\mu}, \tilde{ \phi}^{\zeta}_{\mu} \big\rangle
= \frac{1}{\chi(\mu_{\zeta})^2} \int_M \overline{\tilde{\phi}^{\zeta}_{\mu}(\nu)} \tilde{\phi}^{\zeta}_{\mu} (\nu) h(\nu)\,{\rm d}V(\nu).
\]
Using equation (\ref{b}) and $\langle\psi_i, \psi_j\rangle = \delta_{ij}$ we have
\[
\big\langle \tilde{\phi}^{\zeta}_{\mu}, \tilde{\phi}^{\zeta}_{\mu} \big\rangle
= \frac{1}{\chi(\mu_{\zeta}) ^2} \sum_{i,j,k }\overline{\psi_{i}(\mu)} \psi_j (\mu) \overline{b_{ik}^{\zeta}} b_{jk}^{\zeta}
=  \frac{1}{\chi(\mu_{\zeta}) ^2} \sum_{k} \overline{ \psi_k(\mu_{\zeta})} \psi_k(\mu_{\zeta}) =1.
\]

The proofs of $(2)$--$(7)$ are similar to Theorem~\ref{Kah}.
For instance the proof of $(2)$ goes as follows.
Let $\phi = \sum_{j=1}^m c_j \psi_j$ be the basis expansion of $\phi$. Recall $\overline{\mu_{\zeta}} = \overline{\phi_U( \mu_1 + {\rm i}\zeta \mu_2)}$, then
\begin{align*}
 \big\langle \tilde{\phi}^{\zeta}_{\mu}, \phi \big\rangle
 &=\bigg\langle\frac{s_0(\mu)}{ |s_0(\mu)| \chi(\mu_{\zeta})} \sum_{i,k =1}^m \overline{\psi_i(\mu)} b_{ik}^{\zeta} \psi_{k}, \sum_{j=1}^m c_j \psi_j \bigg\rangle
= \frac{1}{\tau(\mu_{\zeta})} \sum_{i=1}^m \psi_i(\mu) \overline{b_{ik}^{\zeta}} c_k
\\
&= \frac{1}{\tau(\mu_{\zeta})} \sum_{ i=1}^m \overline{\psi_i(\overline{\mu}) b_{ik}^{\zeta}}c_k
=  \frac{1}{\tau(\mu_{\zeta})} \sum_{ k=1}^m \overline{\psi_k(\overline{\mu_{\zeta}})}c_k
 = \frac{\phi(\mu_{\zeta})}{\tau(\mu_{\zeta})}.\tag*{\qed}
\end{align*}
\renewcommand{\qed}{}
\end{proof}

\section[Rawnsley-type coherent and squeezed states on an arbitrary compact smooth manifold]
{Rawnsley-type coherent and squeezed states \\on an arbitrary compact smooth manifold}

\subsection{Rawnsley-type coherent states on an arbitrary compact smooth manifold}

Let $M$ be a compact smooth $l$-dimensional manifold which can be embedded smoothly in ${\mathbb R}^{2l}$ for some $l$ as a consequence of Whitney embedding theorem. Identifying ${\mathbb R}^{2l}$ with ${\mathbb C}^l$ which can be embedded in ${\mathbb C}{\rm P}^n$ by the following map: $(z_1, z_2,\dots,z_l) \in {\mathbb C}^l \rightarrow [z_1,z_2,\dots, z_l, 1] \in {\mathbb C}{\rm P}^n $, where $n = l+1 $. Let $ \epsilon \colon M \rightarrow \epsilon(M) \subset {\mathbb C}{\rm P}^n$ obtained this way.
 $M$ does not need to be symplectic and hence to be prequantized.

Let us define $\Gamma$ to be consisting of $\psi= \epsilon^* (\Phi)$, where $\Phi$ is a holomorphic square integrable (w.r.t.\ to ${\rm d}V_{\rm FS}$, the volume form induced by the Fubini--Study form on ${\mathbb C}{\rm P}^n$) global section of~$H$, $H$ being the hyperplane line bundle on ${\mathbb C}{\rm P}^n$.
In other words, the members of $\Gamma$ are sections of $L =\epsilon^*(H)$. Let $\zeta \in M$. Let ${\rm d}V $ be the volume form on $M$. Let $h(\zeta)\, {\rm d}V(\zeta) = {\rm d} V_{\Sigma}(\epsilon(\zeta))$, where~$h$ is such that all pullback sections are square integrable w.r.t.\ the measure $h {\rm d}V$ on $M$.
Let $\Gamma$ be given an inner product defined as $\langle\phi_1, \phi_2\rangle = \int_{M} \overline{\phi}_1 \phi_2 h {\rm d} V$.
 $\mathcal{H} $ be the subspace of~$\Gamma$ such that $\|\phi\|^2 < \infty$.
Let the dimension of $\mathcal{H} $ be $m$, and $\{\psi_i\}_{i=1}^m$ be a smooth orthornormal basis for $\mathcal{H}$, i.e., $\langle\psi_i, \psi_j\rangle = \int_{M} \overline{\psi}_i \psi_j h {\rm d} V = \delta_{ij}$.

Now we claim that
\begin{Proposition}
$\Gamma = {\mathcal H}$.
\end{Proposition}

\begin{proof}
Let $\psi \in \Gamma$, $\psi = \sum_{k=1}^N c_k \epsilon^*(\Phi_k)$, where $\{\Phi_k\}_{k=1}^N$ is an orthonormal basis of holomorphic square integrable (w.r.t. ${\rm d}V_{\rm FS}$ ) global sections of $H$, the hyperplane bundle on ${\mathbb C}{\rm P}^n$. Then
\[
\int_M | \epsilon^*(\Phi_k)|^2 h {\rm d}V = \int_{M} | \Phi_k(\epsilon(\zeta))|^2 h(\zeta)\, {\rm d}V (\zeta) = \int_{ \epsilon(M)} | \Phi_k (\tau)|^2 {\rm d}V_{\Sigma}(\tau),
\]
where $\tau = \epsilon(\zeta)$ and $h(\zeta)\, {\rm d}V(\zeta) = {\rm d} V_{\Sigma}(\epsilon(\zeta))$. The integral is finite.
Thus $\psi \in {\mathcal H}$.
\end{proof}

Let $\{\psi_i\}_{i=1}^m$ be an orthonormal basis of ${\mathcal H}$ which has dimension $m$.
Let $\phi \in \mathcal{H}$. Then $\phi$ can be expressed as a linear combination of the orthonormal basis elements $\psi_i$, i.e., $\phi = \sum_{i=1}^m \langle\psi_i, \phi\rangle\psi_i $.
Let $\chi^2 = \sum_{i=1}^m |\psi_i|^2 $.

\begin{Proposition}
$\chi \neq 0$.
\end{Proposition}

\begin{proof}
First we note that ${\mathcal H} = \operatorname{Span}(\epsilon^*(\Phi_i))_{i=1}^N $, where $\{\Phi_i\}_{i=1}^N$ is the orthonormal basis for square integrable holomorphic sections of $H$. Now $\epsilon^*(\Phi_i) = \sum_{j=1}^m b^i_j \psi_j$. If $\psi_j (\mu)= 0$ for all $j$, then $\epsilon^*(\Phi_i) (\mu) =0$ for all $i$, or $\Phi_i (\epsilon(\mu)) =0$ for all $i$. But this contradicts the fact that $\{\Phi_i\}_{i=1}^N$ is base point free.
\end{proof}

Let $s_0$ be a fixed global section of $L$ in ${\mathcal H}$ and let $M_0 \subset M$ be the zero set of $s_0$.
For $\mu \in M$ we define

\begin{Definition}
 \[
 \phi_{\mu}= \frac{1}{p(\mu)} \sum_{i=1}^m \overline{f_i(\mu)} \psi_i,
 \]
 where $f_i's $ are functions on $M$ (smooth on $M \setminus M_0$) such that $\psi_i = f_i s_0$ and $p(\mu)^2 = \sum_{i=1}^m |f_i(\mu)|^2$. Note that $p(\mu) \neq 0 $ since $|s_0(\mu)|^2 p^2 (\mu) = \chi(\mu)^2 \neq 0$.
 Then for $\mu \in M \setminus M_0$ one sees that
 \[
 \phi_{\mu}= \frac{s_0(\mu)}{|s_0(\mu)| \chi (\mu)} \sum_{i=1}^m \overline{\psi_i (\mu)} \psi_i.
\]
 \end{Definition}
This is well defined since $\chi \neq 0$ and thus $\frac{s_0 \chi}{|s_0|} \neq 0$ on $M \setminus M_0$.
Let $\tau = \frac{s_0 \chi}{|s_0|}$.

\begin{Proposition}%\label{section}
$\phi_{\mu} \in {\mathcal H}$, i.e., it is a section of $L$.
\end{Proposition}

\begin{proof}
Let $\psi_i = f_i s_0$, where $f_i$ is a function on $M$, which is smooth away from zeroes of~$s_0$.
By~definition $\phi_{\mu}= \frac{1}{p(\mu)} \sum_{i=1}^m \overline{f_i(\mu)} \psi_i $, where $p(\mu)^2 = \sum_{i=1}^m |f_i(\mu)|^2$. Since $ \frac{\overline{f}_i}{p}$ is a smooth function, $\phi_{\mu}$ is a section.
\end{proof}

One can check that $\|\phi_{\mu}\|^2 =1 $ and $ \phi_{\mu}(\mu) = \frac{s_0(\mu)}{|s_0(\mu)|} \chi(\mu) = \tau(\mu)$ on $M \setminus M_0$.
Then, as in Theorem~\ref{Kah}, we have the following result.

\begin{Theorem}%\label{pbcoh}
Let $M$ be a smooth compact manifold with $\mathcal{H}$, $s_0$, $\phi_{\mu}$ and $\tau$ defined above. Then, for all $\mu \in M$, we have
\begin{enumerate}\itemsep=0pt

\item[$(a)$] $\phi_{\mu}$ are Rawnsley-type coherent states of $M$.
In fact, for $\mu \in M$ the general formula is $\langle\phi_{\mu}, \phi\rangle = \frac{g(\mu)}{ p(\mu)}$, where $\phi = g s_0$.
 Also, $\langle\phi_{\mu}, \phi\rangle = \frac{\phi(\mu)}{ \tau(\mu)} $ for $\mu \in M \setminus M_0$.

\item[$(b)$] $\phi_{\mu} $ satisfy properties $(b)$ of Theorem~$\ref{Kah}$.
\end{enumerate}
\end{Theorem}

\begin{proof}
The proof is exactly the same as in Theorem~\ref{Kah}.
\end{proof}

\begin{Remark}Note that in the smooth category, the evaluation functional need not be continuous, since the topology on the space of sections is the $L^2$ topology.
But in the case we restrict ourselves to pullback of holomorphic sections it works since the evaluation functional is continuous in the holomorphic category.
\end{Remark}

\subsection{Squeezed states for compact smooth manifolds}

Let $M$ be an even-dimensional smooth manifold of dimension $2d$ which is embedded in ${\mathbb C}{\rm P}^n$ by an embedding $\epsilon$. Let $\Sigma = \epsilon(M) \subset {\mathbb C}{\rm P}^n$. Let $\Gamma$, ${\mathcal H}$, ${\psi_i}$ etc be defined as in the previous section.

Let $q \in U \subset M$, where $U$ is an open neighbourhood of $q$ such that there exists $\phi_U $, a~homeomorphism to an open ball $V \subset {\mathbb C}^d$ to~$U$. Let $q=\nu= \phi_U( \nu_1 + {\rm i}\nu_2) $, $\nu_1 + {\rm i}\nu_2 \in V$, and $\zeta \in {\mathbb R}$ be such that $ \nu_1 + {\rm i}\zeta \nu_2 $ belongs to $V$. Then $\nu_{\zeta} = \phi_U( \nu_1 + {\rm i}\zeta \nu_2) \in U$. Let $q_{\zeta} = \nu_{\zeta}$, then there is a~subset $\tilde{M}$ of $M$ of dimension at most $(2d-1)$ such that $M \setminus \tilde{M}$ can be covered by a single open set $U$ of the above kind.
This follows from Proposition~\ref{celldecomp}.

Let $M$ be a $2d$-dimensional smooth compact manifold, $\tilde{M}$ be as in Proposition~\ref{celldecomp},
 $\{\psi_i\}_{i=1}^m$ be the orthonormal basis for the Hilbert space as described in the previous section, with inner product $\langle\cdot, \cdot\rangle$, and
$\mu = \phi_U (\mu_1 + {\rm i}\mu_2)$ and $\mu_{\zeta} = \phi_U(\mu_1 + {\rm i}\zeta \mu_2)$. We define
$\chi(\mu_{\zeta}) ^2 = \sum_{i=1}^m | \psi_i(\mu_{\zeta})|^2 $ which is non-zero since $\{\psi_i\}$ is a base-point free basis.

Let $s_0$ be a fixed element of ${\mathcal H}$ which vanishes on an subset $M_0$ of $ M$. Then $\tilde{M} \cup M_0$ is of measure zero in $M$.
As before let $\tau (\mu_{\zeta}) = \frac{s_0(\mu_{\zeta}) \chi(\mu_{\zeta})}{|s_0(\mu_{\zeta})|}$.
For $\mu_{\zeta} \in M \setminus \big(\tilde{M} \cup M_0\big)$ we define the squeezed states as follows.

\begin{Definition}
\[
\phi^{\zeta}_{\mu} (\nu) = \frac{s_0(\mu_{\zeta})}{|s_0(\mu_{\zeta})| \chi(\mu_{\zeta})} \sum_{i=1}^m \overline{\psi_i(\mu_{\zeta})} \psi_i(\nu).
\]
\end{Definition}

Then as before we have the following theorem.

\begin{Theorem}%\label{Sq1}
The squeezed states $\phi^{\zeta}_{\mu}$ satisfy the properties
$(1)$--$(7)$ of Theorem $\ref{Sq2}$.
\end{Theorem}

\begin{proof}
The proof is similar to Theorem~\ref{Kah}.
\end{proof}

\begin{Example}[example of the Lobachevsky plane, unit disc $D = \{z\colon |z|<1\}$]
For non-compact manifolds the Hilbert space of quantization is usually infinite-dimensional. We give the example of the Lovachevsky plane.

The quantization on the Lobachevsky plane has been explained in Perelomov \cite[Chapter~16]{Pe}.
Consider the space of functions analytic in the domain $D$ with the inner product
\[
\langle f,g\rangle = \bigg(\frac{1}{\hbar} -1\bigg) \int \overline{f}(z) g(z) \big(1- |z|^2\big)^{\frac{1}{\hbar}}\, {\rm d}\mu (z, \overline{z}),
\]
where ${\rm d}\mu(z, \overline{z}) = \frac{1}{2 \pi {\rm i}} \frac{{\rm d}z \wedge {\rm d} \overline{z}}{ (1 - |z|^2)^2}$. The Hilbert space ${\mathcal H}$ consists of square integrable analytic functions with respect to this inner product. The orthonomal basis is given in \cite{Pe}, for instance.
In fact $\psi_i (z) = (i!)^{-1/2} \big[\big(\frac{1}{\hbar}\cdots \big(\frac{1}{\hbar} -1 + i\big)\big] ^{1/2} z^i$ are an orthonormal basis for the Hilbert space ${\mathcal H}$.
Notice $\overline{\psi_i(z)} = \psi_i (\overline{z})$ for all~$i$.

Let $\mu_{\zeta} = \mu_1 + {\rm i}\zeta \mu_2$ and $\nu_{\zeta} = \nu_1 + {\rm i}\zeta \nu_2$, $\zeta \in {\mathbb R}$ such that $\mu_{\zeta}$ and $ \nu_{\zeta}$ belong to $D$.
Then, according to our definition (adapted to the infinite-dimensional Hilbert space), the squeezed states are $\phi^{\zeta}_{\mu}$ and $\tilde{\phi}^{\zeta}_{\mu}$, where
$\phi^{\zeta}_{\mu} (\nu) = \frac{1}{\chi(\mu_{\zeta})} \sum_{i=1}^\infty \overline{\psi_i(\mu_{\zeta})} \psi_i(\nu)$ and $\tilde{\phi}^{\zeta}_{\mu} (\nu)= \frac{1}{\chi(\mu_{\zeta})} \sum_{i=1}^\infty \overline{\psi_i(\mu)} \psi_i(\nu_{\zeta})$.
One can show that both series converge.
\end{Example}

\section[Berezin quantization of pullback operators on totally real submanifolds of open sets in CP\textasciicircum{}n]
 {Berezin quantization of pullback operators on totally\\ real submanifoldsof open sets in $\boldsymbol{{\mathbb C}{\rm P}^n}$}

${\mathbb C}{\rm P}^n$ is a compact homogeneous K\"ahler manifold.
Berezin quantization of ${\mathbb C}{\rm P}^n$ is described in~\cite{Be}.
Let $\Sigma$ be a compact smooth manifold.
Supposing ${\mathcal U}$ is an open subset of ${\mathbb C}{\rm P}^n$. Suppose $ \epsilon\colon \Sigma \rightarrow {\mathcal U} $ is a smooth embedding such that $\epsilon(\Sigma) \subset {\mathcal U} $ is a submanifold which is totally real and of real dimension $n$. The conditions ensure that the identity theorem holds, namely if two holomorphic functions on ${\mathcal U}$ agree on $\epsilon(\Sigma)$ they are identical on ${\mathcal U}$. Any other submanifold for which the criterion for the identity theorem holds will also exhibit this type of Berezin quantization, see for instance \cite[Chapter~9, Lemma~2]{Bo}.

Let $H $ be the hyperplane bundle on ${\mathbb C}{\rm P}^n$ restricted to ${\mathcal U}$. Let $L = \epsilon^*(H)$ be a line bundle on $\Sigma$ with an inner product.
Let $\mathcal{H}_\Sigma$, the Hilbert space on $\Sigma$ given by the space of square integrable sections $f$ of $L$, where $f= \epsilon^{*}\big(\tilde{f}\big)$ and~$\tilde{f}$ is a holomorphic, local (defined on the open set ${\mathcal U} \subset {\mathbb C}{\rm P}^n$), square integrable section of the hyperplane bundle on~$\mathbb{C}P^n$. We will think of $\tilde{f}$ as a holomorphic function defined on ${\mathcal U}$.

\begin{Note}Since $\epsilon(\Sigma)$ is such that the criterion of the identity theorem holds, there is a unique $\tilde{f}$ such that $f= \epsilon^{*}\big(\tilde{f}\big)$. For if $\tilde{g}$ be another holomorphic function defined on ${\mathcal U}$ such that $f= \epsilon^{*}(\tilde{g})$, then $\tilde{f}$, $\tilde{g}$ agree on~$\epsilon(\Sigma)$, and hence $\tilde{f}=\tilde{g}$ on~${\mathcal U}$, since they are holomorphic, by the identity theorem.
\end{Note}

Let $\widehat{a}$ be a linear bounded operator acting on $\mathcal{H}_\Sigma$ such that there exists an operator $\widehat{A}$ acting on holomorphic square integrable sections~$\tilde{f}$ on~$\mathbb{C}P^n$ with the property that
$ \widehat{a}(f)= \widehat{a} \big(\epsilon^{*}\big(\tilde{f}\big)\big)=\epsilon^{*}\widehat{A}\big(\tilde{f}\big)$ and~$\widehat{A}$ takes square integrable holomorphic sections to square integrable holomorphic sections on ${\mathcal U}$.

 If such an operator $\widehat{A}$, exists, it is unique. Indeed, if there are two such operators $A$ and $B$ then
$\epsilon^{*}\widehat{A}\big(\tilde{f}\big) = \epsilon^{*}\widehat{B}\big(\tilde{f}\big)$ for all $\tilde{f}$, holomorphic section on~${\mathcal U}$. In other words, on $\epsilon(\Sigma)$, $ \big(\widehat{A}- \widehat{B}\big)\big(\tilde{f}\big)=0$ for all holomorphic sections $\tilde{f}$.
Since $\big(\widehat{A}- \widehat{B}\big)\big(\tilde{f}\big)$ is holomorphic and zero on $\epsilon(\Sigma)$, it is identically zero (since $\epsilon(\Sigma)$ is totally real and of real dimension~$n$). Thus $\widehat{A} = \widehat{B}$.

 \begin{Definition}
For $p, q \in \Sigma$, let us define the ${\mathbb C}{\rm P}^n$-symbol of $\widehat{a}$ as
\[
\widehat{a}(p,q) = A\big(\epsilon(p), \overline{\epsilon{(q)}}\big) = \frac{\big\langle\tilde{\Phi}_{\overline{\epsilon(p)}}, \widehat{A} \tilde{\Phi}_{\overline{\epsilon(q)}}\big\rangle_{{\mathbb C}{\rm P}^n}}{\big\langle\tilde{\Phi}_{\overline{\epsilon(p)}}, \tilde{\Phi}_{\overline{\epsilon{(q)}}}\big\rangle_{{\mathbb C}{\rm P}^n}},
\]
where $\langle\cdot, \cdot\rangle_{{\mathbb C}{\rm P}^n}$ is the $\langle\cdot\mid \cdot\rangle$ used in \cite[equation (16.5.13)]{Pe} and $\tilde{\Phi}_{\overline{\zeta}}$ is the coherent state on~${\mathbb C}{\rm P}^n$ parametrized by $\zeta \in {\mathbb C}{\rm P}^n$.
Note that the integral in the inner product is taken on entire~${\mathbb C}{\rm P}^n$ and not on~$\epsilon(\Sigma)$.
 It is the symbol of the operator $\widehat{A}$ evaluated at $\big(\epsilon(p), \overline{\epsilon{(q)}}\big)$ (see~\cite{Be} and \cite[equation~(16.5.14)]{Pe}). Thus it is named as the ${\mathbb C}{\rm P}^n$-symbol of $\widehat{a}$.
\end{Definition}

 The operator $\widehat{a}$ can be derived from the ${\mathbb C}{\rm P}^n$-symbol of $\widehat{a}$:
\begin{align*}%\label{hat(a)}
\widehat{a}f(p)&= \widehat{a}\big(\epsilon^{*}\circ \tilde{f}\big)(p)= \epsilon^{*}\widehat{A}\big(\tilde{f}\big)(p)
= \big( \widehat{A}\tilde{f}\big)(\epsilon(p))
\\
&= c(\hbar) \int_{{\mathbb C}{\rm P}^n} A\big(\epsilon(p), \overline{\zeta}\big)\tilde{f}(\zeta) L_{\hbar}\big(\epsilon(z),\overline{\zeta}\big) \exp\bigg({-}\frac{1}{\hbar} F\big(\zeta,\overline{\zeta}\big)\bigg){\rm d}\mu\big(\zeta,\overline{\zeta}\big),
\end{align*}
where
$L_{\hbar}\big(\eta, \overline{\zeta}\big) = \tilde{\Phi}_{\overline{\zeta}} (\eta)$ (see~\cite{Be, Pe} for more details).

 If we choose $\widehat{B}$ instead, the integral will not change \big(since $\epsilon^{*}\widehat{A}\big(\tilde{f}\big)(p) = \epsilon^{*}\widehat{B}\big(\tilde{f}\big)(p) $ for all~$\tilde{f}$ holomorphic on~${\mathbb C}{\rm P}^n$\big).
Let $\widehat{a}_1$ and $\widehat{a}_2$ be two bounded linear operators on ${\mathcal H}$ such that $\widehat{a}_i(f) = \epsilon^*\big(\widehat{A}_i\big)\big(\tilde{f}\big) $, $i=1,2$.

 \begin{Definition}
 One can define the star product $(a_1\star a_2)(p,p)= (A_1\star A_2)\big(\epsilon(p),\overline{\epsilon(p)}\big)$.
 \end{Definition}

 \begin{Proposition}
 $(a_1\star a_2)(p,p)$ is the ${\mathbb C}{\rm P}^n$-symbol of $\widehat{a}_1 \circ \widehat{a}_2$.
 \end{Proposition}

 \begin{proof}
 Let $\widehat{a} = \widehat{a}_1 \circ \widehat{a}_2$. Let $\widehat{A}$, $\widehat{A}_1$, $\widehat{A}_2$ be the corresponding operators for $\widehat{a}$, $\widehat{a}_1$, $\widehat{a}_2$ respectively. It~is easy to show that $\widehat{A} = \widehat{A}_1 \circ \widehat{A}_2$.
 It follows from Berezin \cite{Be} that $(A_1 \star A_2)\big(\epsilon{(p)}, \overline{\epsilon{(p)}} \big)$ is the ${\mathbb C}{\rm P}^n$-symbol of $ \widehat{A}_1 \circ \widehat{A}_2 $.
 Thus we are done.
 \end{proof}

 Let $F$, $G$ be smooth functions on $\Sigma$ such that $F = \epsilon^*\big(\tilde{F}\big)$, $G = \epsilon^*\big(\tilde{G}\big)$, where  $\tilde{F}$, $\tilde{G}$ are local holomorphic functions on ${\mathcal U} \subset {\mathbb C}{\rm P}^n$. Recall $\tilde{F}$, $\tilde{G}$ are unique, since $\epsilon(\Sigma)$ satisfies the criterion of the identity theorem.
 Let
 \[
 \{F, G\}_{\rm PB} \equiv \sum_{ij} \Omega^{ij}_{\rm FS} \bigg(\frac{\partial \tilde{F} }{\partial z_i} \frac{\partial \tilde{G} }{\partial \overline{z}_i} - \frac{\partial \tilde{F} }{\partial \overline{z}_i}\frac{\partial \tilde{G} }{\partial z_i}\bigg)\bigg|_{\epsilon(\Sigma)}
  \]
  be a Poisson bracket on such functions $F$, $G$ on $\Sigma$, where $\Omega^{ij}$ is the inverse matrix to that of the Fubini--Study form.

 \begin{Proposition}
 The star product defined above satisfies the correspondence principle.
 \end{Proposition}

 \begin{proof}
By \cite{Be} (see also \cite[Chapter~16.5]{Pe})
\[
\lim_{\hbar \to 0} (a_1\star a_2)(p, p)= A_1\big(\epsilon(p),\overline{\epsilon(p)}\big)A_2\big(\epsilon(p),\overline{\epsilon(p)}\big),
\]
and
\begin{gather*}
\lim_{\hbar \to 0} \frac{1}{\hbar}((a_1\star a_2)(p,p)- (a_2\star a_1)(p,p))
\\ \qquad
{}=\lim_{\hbar \to 0} \frac{1}{\hbar}\big((A_1\star A_2)\big(\epsilon(p),\overline{\epsilon(p)}\big)- (A_2\star A_1)\big(\epsilon(p),\overline{\epsilon(p)}\big)\big)
\\ \qquad
{}=\iota\big\{A_1\big(\epsilon(p),\overline{\epsilon(p)}\big), A_2\big(\epsilon(p),\overline{\epsilon(p)}\big)\big\}_{\rm FS},
 \end{gather*}
where $\{\,,\,\}_{\rm FS}$ stands for the Poisson bracket on ${\mathbb C}{\rm P}^n$ induced by the Fubini--Study form $\Omega_{\rm FS}$.  Thus
\[
\lim_{\hbar \to 0} \frac{1}{\hbar}((a_1\star a_2)(p,p)- (a_2\star a_1)(p,p)) =
 \iota \{a_1, a_2\}_{\rm PB}.
 \tag*{\qed}
\]
\renewcommand{\qed}{}
\end{proof}

\begin{Remark} There is nothing special about ${\mathbb C}{\rm P}^n$ in this type of Berezin quantization. If~we can embed $\Sigma$ (of real dimension~$n$) in a general homogeneous K\"ahler manifold~$M$ (of real dimension~$2n$ and which has Berezin quantization) as a totally real submanifold (or if the embedding satisfies the criterion of the identity theorem), then we can define the $M$-symbol of operators of the above type and define the star product analogously such that the correspondence principle is satisfied. Then $\Sigma$ has the above type of Berezin quantization (which depends on the embedding).
 \end{Remark}
 \begin{Remark}The pullback coherent states are useful in defining Berezin quantization of arbitrary smooth submanifolds of ${\mathbb C}{\rm P}^n$ under certain conditions. This is work in progress.
 \end{Remark}

\appendix

\section[Review of Rawnsley and Perelomov coherent states on CP\textasciicircum{}n]
{Review of Rawnsley and Perelomov coherent states on $\boldsymbol{{\mathbb C}{\rm P}^n}$}

Let ${\mathbb C}{\rm P}^n$ with $k \Omega_{\rm FS}$, $k$ is a positive integer and $\Omega_{\rm FS}$ is the Fubini--Study form. Let us consider the geometric quantization of ${\mathbb C}{\rm P}^n $, where the quantum line bundle is $H^k$, $H$ being the hyperplane line bundle.

${\mathbb C}{\rm P}^n$ is to be thought of as a coadjoint orbit ${\mathcal O}_p$. In fact $ {\mathcal O_p}$ is diffeomorphic to $ \frac{{\rm SU}(n+1)}{{\rm S}({\rm U}(n) \times {\rm U}(1))}$. (For a proof, one can adapt the proof for the Grassmannian in \cite[p.~183]{N} and note that the inclusion of ${\rm SU}(n+1)/{\rm S}({\rm U}(n)\times {\rm U}(1))$ in ${\rm U}(n+1)/{\rm U}(n)\times {\rm U}(1)$ is a~diffeomorphism. This can be seen by the fact that both the quotients have the same dimension. Here ${\rm S}({\rm U}(n)\times {\rm U}(1))$ denotes the collection of pairs $(x,y)$, $x \in {\rm U}(n)$, $y \in {\rm U}(1)$ such that $\det x \cdot \det y = 1$.)

Let $p \in {\mathcal O}_p \subset X = {\mathfrak{su}}(n+1)^*$. Given $\zeta \in {\mathfrak{su}}(n+1)$, there is a natural function $\lambda_{\zeta}(p) = \langle\zeta, p\rangle$. Let $\widehat{\lambda}_{\zeta}$ be the geometric quantization operator acting on sections of $H^k$. Then by a result of Kostant \cite{K}, we have $\zeta \rightarrow \widehat{\lambda}_{\zeta}$ is a representation of the Lie algebra of ${\rm SU}(n+1)$ and exponentiating we get a unitary representation of ${\rm SU}(n+1)$.

The fact (implicit in Kostant's paper) that $\zeta \rightarrow \widehat{\lambda}_{\zeta}$ is a representation can be seen as follows.
Let $\Theta = \Theta^{1,0} + \Theta^{0,1}$ be the K\"ahler potential of the Fubini--Study form $\Omega_{\rm FS}$.
Let $M$ be a coadjoint orbit of a Lie group with integral K\"ahler form.
Let $\tau_1$ be a function on $M$ and let $\xi^1$ be its Hamiltonian vector field. Then{\samepage
\[
\widehat{\tau}_1(\psi)= -{\rm i}\hbar\big[\xi^1(\psi)-{\rm i}\big(\xi^1\lrcorner\Theta\big)\psi\big] + \tau_1\psi,
\]
where $\Theta$ is a K\"ahler potential of the K\"ahler form on the coadjoint orbit.}

The first term,
\[
\big[\xi^1(\psi)-{\rm i}\big(\xi^1\lrcorner\Theta\big)\psi\big]
=\big[\big(\partial+ \overline{\partial}\big)\psi\big(\xi^1\big)-{\rm i}\big(\Theta^{1,0}+ \Theta^{0,1}\big)\big(\xi^1\big)\psi\big].
\]
But $\psi$ satisfies the polarization equation $\big(\overline{\partial}-{\rm i}\Theta^{0,1}\big)\psi = 0$, so that
\[
\overline{\partial}\psi- {\rm i}\sum_ig_i{\rm d}\overline{z}_i \psi= 0,
\]
locally. Thus $ \psi= {\rm e}^{-\sum g_i} h(z)$, where $h(z)$ is holomorphic and $\Theta^{0,1}= {\rm i} \sum_i g_i {\rm d}\overline{z_i}$. Only the first term in $\widehat{\tau}_1$, namely $\big[\partial \psi\big(\xi^1\big)-{\rm i}\Theta^{1,0}\big(\xi^1\big)\psi\big]$ survives.

Let
\[
\Theta = {\rm i} \bigg[\sum_ih_i{\rm d}z_i + \sum_ig_i{\rm d}\overline{z}_i\bigg] = \Theta^{1,0}+\Theta^{0,1},
\]
where
\[
\Theta^{1,0} = {\rm i} \sum_i h_i{\rm d}z_i\qquad \text{and}\qquad \Theta^{0,1}= {\rm i} \sum_ig_i{\rm d}\overline{z}_i.
\]
Since $\Theta$ is a unitary connection, we have
$\overline{\Theta^{1,0}}= - \Theta^{0,1}$.
Now
\[
\overline{\Theta^{1,0}}=-{\rm i}\sum_i\overline{h}_i{\rm d}\overline{z}_i =-{\rm i}\sum_ig_i{\rm d}\overline{z}_i =- \Theta^{0,1}.
\]
Thus $\overline{h}_i = g_i$ and hence $\Theta^{1,0} = {\rm i} \sum_i \overline{g} _i{\rm d}z_i$.
We have
\[
\partial\psi- {\rm i}\Theta^{1,0}\psi
={\rm e}^{-\sum g_i}\Big[\sum (-\partial g_i+\overline{g}_i {\rm d}z_i)h(z)+ \partial h(z)\Big].
\]

The geometric quantization operator acts as follows
\[
\widehat{\tau}_1(\psi)
=\big(\partial \psi- {\rm i}\Theta^{1,0}\psi\big)\big(\xi^1\big)+ \tau_1\psi
= {\rm e}^{-\sum g_i}[T_1(h(z))],
\]
where
\[
T_1 h(z) = \sum(-\partial g_i+\overline{g}_i{\rm d}z_i)\big(\xi^1\big)h(z)+ \partial h\big(\xi^1\big) + \tau_1 h(z).
\]

Suppose $\tau_1$, $\tau_2$ be two functions on the coadjoint orbit.
Then by definition of $\widehat{\tau}_i$, $i=1,2$, we have $[\widehat{\tau}_1,\widehat{\tau}_2]=-{\rm i} \widehat{\{\tau_1,\tau_2\}}$ \cite{W}.
We wish to see the action on the holomorphic part of~$\psi$, namely~$h(z)$.
Let us denote by $ \widehat{\widehat{\tau}}_1 =T_1$ and
$ \widehat{\widehat{\tau}}_2=T_2$, where
\begin{gather*}
T_1( h) = \sum G\big(\xi^1\big)h + \partial h\big(\xi^1\big) +\tau_1 h,
\\
T_2 (h) = \sum G\big(\xi^2\big)h + \partial h\big(\xi^2\big) +\tau_2 h,
\\
G= \sum (-\partial g_i+\overline{g}_i{\rm d}z_i).
\end{gather*}
Then,
\[
\widehat{\tau}_1(\psi)= {\rm e}^{-\sum g_i}\big(\widehat{\widehat{\tau}}_1(h)\big)= {\rm e}^{-\sum g_i}T_1(h),\qquad
\widehat{\tau}_2(\psi)= {\rm e}^{-\sum g_i}\big(\widehat{\widehat{\tau}}_2(h)\big)= {\rm e}^{-\sum g_i}T_2(h).
\]
We want to show that $\big[\widehat{\widehat{\tau}}_1, \widehat{\widehat{\tau}}_2\big] =-{\rm i}\widehat{\widehat{\{\tau_1,\tau_2\}}}$:
\[
\big[\widehat{\widehat{\tau}}_1, \widehat{\widehat{\tau}}_2\big]h =[T_1, T_2]h= (T_1T_2 - T_2T_1)h.
\]
Thus
$T_1(h)= {\rm e}^{\sum g_i} \widehat{\tau}_1(\psi)$ and
$T_2(h)= {\rm e}^{\sum g_i} \widehat{\tau}_2(\psi)$.
Hence
\[
T_1T_2(h)- T_2T_1(h)= T_1\big({\rm e}^{\sum g_i} \widehat{\tau}_2(\psi)\big)- T_2\big({\rm e}^{\sum g_i} \widehat{\tau}_1(\psi)\big).
\]
Let $h_2$ be such that $\widehat{\tau}_2(\psi)= \psi_2 = {\rm e}^{-\sum g_i}h_2 $, then
${\rm e}^{\sum g_i} \widehat{\tau}_2(\psi)= h_2$,
$T_1(h_2)= {\rm e}^{\sum g_i} \widehat{\tau}_1(\psi_2)$ and
$T_2(h_1)= {\rm e}^{\sum g_i} \widehat{\tau}_2(\psi_1)$.
Then
\begin{align}
 \big[\widehat{\widehat{\tau}}_1, \widehat{\widehat{\tau}}_1\big] &= [T_1, T_2](h) = T_1(h_2)- T_2(h_1)
= {\rm e}^{\sum g_i} \widehat{\tau}_1(\psi_2)-{\rm e}^{\sum g_i} \widehat{\tau}_2(\psi_1)\nonumber
\\
& = {\rm e}^{\sum g_i} \widehat{\tau}_1\widehat{\tau}_2(\psi)-{\rm e}^{\sum g_i} \widehat{\tau}_2\widehat{\tau}_1(\psi)
= {\rm e}^{\sum g_i} \big(\widehat{\tau}_1\widehat{\tau}_2(\psi)-\widehat{\tau}_2\widehat{\tau}_1(\psi)\big)
 =-{\rm i}{\rm e}^{\sum g_i}\big(\widehat{\{\tau_1,\tau_2\}}(\psi)\big)\nonumber
 \\
 &  =-{\rm i}{\rm e}^{\sum g_i}(\widehat{P}(\psi)) =-{\rm i}\widehat{\widehat{P}}(h),\label{comm}
\end{align}
 where $P=\{\tau_1,\tau_2\}$, the Poisson bracket.

Now, by the coadjoint action definition,
$(\lambda_i, g\cdot p) = (g\cdot \lambda_i, p)$, where $\lambda_i$ are the generators of the Lie algebra.
Let
\begin{gather*}
g_i= {\rm e}^{t\lambda_i^{\#}}\qquad \text{and}\qquad \tau_i= (\lambda_i, p),
\\
\lambda_i^{\#}(\tau_j)= \lambda_i^{\#}(\lambda_i, p)
=\frac{\rm d}{{\rm d}t} (\lambda_j, g_i\cdot p)\bigg|_{t=0}
=\bigg(\frac{\rm d}{{\rm d}t}{\rm e}^{t\lambda_i^{\#}}\lambda_j {\rm e}^{-t\lambda_i^{\#}}\bigg|_{t=0}, p\bigg)
=([\lambda_i, \lambda_j], p).
\end{gather*}
Now
\[
\omega\big(X_{\tau_j}, \lambda_i^{\#}\big)=- {\rm d}\tau_j\big(\lambda_i^{\#}\big)
=-\lambda_i^{\#}(\tau_j)
=\big([\lambda_j, \lambda_i], p\big)
=\omega\big(\lambda_j^{\#},\lambda_i^{\#}\big).
\]
Therefore $X_{\tau_i}= \lambda_i^{\#}$, $i$ runs over the generator indices.

Suppose $\big[\lambda_i^\#,\lambda_j^\#\big] = \sum a_{i,j}^k \lambda_k^\#$.
Then $[X_{\tau_i}, X_{\tau_j}]= \sum a_{i,j}^k X_{\tau_k}= X_{\{\tau_i,\tau_j\}}$.
Therefore
\[
X_{\{\tau_i,\tau_j\}} = \sum a_{i,j}^k X_{\tau_k} = X_{(\sum a_{i,j}^k \tau_k)}\qquad \text{and}\qquad \{\tau_i,\tau_j\} =\sum a_{i,j}^k \tau_k.
\]
Then we have $[\widehat{\tau}_i, \widehat{\tau}_j]=-{\rm i}\widehat{\{\tau_i,\tau_j\}}=-{\rm i}\sum a_{i,j}^k \widehat{\tau}_k$.
Similarly, $[\widehat{\widehat{\tau}}_i,\widehat{\widehat{\tau}}_j]=-{\rm i}\sum a_{i,j}^k \widehat{\widehat{\tau}}_k$, by equa\-tion~(\ref{comm}).
Thus we have the commutation relation for $\widehat{\tau}_i$, $\widehat{\tau}_j$ and $\widehat{\widehat{\tau}}_i$, $\widehat{\widehat{\tau}}_j$ as the generators $\lambda_i$, $\lambda_j$ upto a factor of $-{\rm i}$.
If we write $\widehat{\chi}_j = {\rm i} \widehat{\tau}_j$, where $j$ runs over the generators of the Lie algebra, then
$[\widehat{\chi}_i, \widehat{\chi}_j]= \sum a_{i,j}^k \widehat{\chi}_k$, i.e., they satisfy the same commutation relation as the Lie algebra.
Similarly, $[\widehat{\widehat{\chi}}_i,\widehat{\widehat{\chi}}_j]= \sum a_{i,j}^k \widehat{\widehat{\chi}}_k$, the same commutation relation as the Lie algebra.

\subsection[Perelomov coherent states for CP\textasciicircum{}n]
{Perelomov coherent states for $\boldsymbol{{\mathbb C}{\rm P}^n}$}

Let the Hilbert space for quantization of ${\mathbb C}{\rm P}^n$ be identified with holomorphic sections of the hyperplane bundle, as before. ${\mathbb C}{\rm P}^n = \frac{{\rm SU}(n+1)}{{\rm S}({\rm U}(n) \times {\rm U}(1))}$ is a coadjoint orbit.
 A point $A$ in ${\mathbb C}{\rm P}^n$ is identified with a projection operator $P_A$ as before. The generators of ${\mathfrak{su}}(n+1)$ give rise to Hamiltonian functions $ \chi_j (A) = {\rm i} \tau_j(A) ={\rm i} \operatorname{Tr}\big(P_A \lambda_j^{\rm T}\big) $, $j$ runs over the generators. Then the geometric quantization operators $\widehat{\chi }_j$ corresponding to the Fubini--Study form on ${\mathbb C}{\rm P}^n$ give a~representation of ${\mathfrak{su}}(n+1)$ on the Hilbert space of geometric quantization, in accordance with what we showed above (implicit in~\cite{K}).

Then we have the well known result that
exponentiating the above representation gives rise to a unitary representation of ${\rm SU}(n+1)$ on the Hilbert space~\cite{K}.
Take a fiducial vector $\Psi_0$ which is invariant under ${\rm S}({\rm U}(n) \times {\rm U}(1))$ and let the group elements act on it. This gives the Perelomov coherent states
$g \cdot \Psi_0 = {\rm e}^{{\rm i} \alpha} \Psi_{(g \cdot p)}$, where $p$ is an element of the coadjoint orbit and~${\rm e}^{{\rm i} \alpha}$ is a phase factor \cite{Pe}.

\subsection[Rawnsley coherent states on CP\textasciicircum{}n]
{Rawnsley coherent states on $\boldsymbol{{\mathbb C}{\rm P}^n}$}

Let $L$ be the line bundle obtained by geometric quantization (namely the hyperplane bundle) and $L_0= L\setminus \{0\}$, $L$ with zero section removed.

We have shown that
\[
\phi_{\mu}= \frac{1}{p(\mu)}\sum \overline{f_i(\mu)}\psi_i\qquad \text{and}\qquad
\langle\phi_{\mu}, \psi\rangle = \frac{\psi(\mu)}{\tau(\mu)},
\]
where
\[
\tau(\mu)= \frac{s_0(\mu)}{|s_0(\mu)|} \chi (\mu), \qquad
\chi(\mu)= \Big(\sum |\psi_i(\mu)|^2\Big)^{\frac{1}{2}}.
\]
Let $a(\mu)= \frac{\chi(\mu)}{|s_0(\mu)|}$. Then $a(\mu)$ is a smooth function where $s_0(\mu)\neq 0$, $\mu\in M$.
Let $\tau(\mu)= s_0 (\mu) a(\mu)$,
 $e_q:=\overline{a(\mu)} \phi_{\mu}$, and $q := s_0(\mu) \in \pi^{-1}(\mu)$.
Then as in \cite{Ra}, we have $\langle e_q, \psi\rangle\cdot q= \psi(\mu)$.

Let $q^{\prime}\in \pi^{-1}(\mu)$ such that
\[
q^{\prime}= c \cdot q,\qquad
e_{q^{\prime}}= \overline{c}^{-1}\cdot e_q,\qquad
\langle e_{q^{\prime}}, \psi\rangle\cdot q^{\prime} = \langle\overline{c}^{-1} e_q, \psi\rangle\cdot cq = \overline{c}^{-1}\cdot c \langle e_q, \psi\rangle\cdot q = \psi(\mu).
\]

We give an explicit local description of Rawnsley coherent states on ${\mathbb C}{\rm P}^n$.
Let $U $ be an open subset of ${\mathbb C}{\rm P}^n$ given by $\{z_0 \neq 0\}$, where $[z_0,\dots,z_n]$ is a coordinate system on $ {\mathbb C}{\rm P}^n$. Let $(\mu_1, \mu_2,\dots, \mu_n)$ be coordinates on $U$ such that $\mu_i = z_i/z_0$, $i=1,\dots,n$.
Let $s_0$ be a holomorphic section of $H$ (the hyperplane bundle on $\mathbb{C}P^n$) such that
\[
\int_{U} \frac{|s_0(\mu)|^2}{(1+|\mu|^2)^{2s}}{\rm d}\mu_1 {\rm d}\mu_2\cdots {\rm d}\mu_n =1.
\]
Let
\[
c_p(\mu)= \frac{1}{\sum |\mu_1|^{2p_1} |\mu_2|^{2p_2}\cdots |\mu_n|^{2p_n}},
\]
where the sum runs over $p_1+p_2+\dots +p_n=p$, $p=0,1,\dots,2s$.
Let
\[
\Phi_{(p_1,p_2,\dots,p_n;p)}(\mu) = \sqrt{c_p(\mu)} \frac{\mu_1^{p_1}\cdots \mu_n^{p_n}}{(1+|\mu|^2)^s} s_0(\mu),
\]
where $p_1+ \dots +p_n = p$.
These form an orthonormal basis for sections of $H$ when restricted to~$U$.

The Rawnsley coherent states are given on $U$ by $\psi_{\mu}$ reading as follows
\[
\psi_{\mu} (\nu) := \big(1+|\mu|^2\big)^{2s} \sum_{{p_1+p_2+\dots +p_n=p;\,p=0,1,\dots,2s}} \overline{\Phi_{(p_1,p_2,\dots,p_n;p)}(\mu)} \Phi_{(p_1,p_2,\dots,p_n;p)}(\nu).
\]
Then $\psi_{\mu}(\mu)= |s_0(\mu)|^2$ and $\langle \psi_{\mu}, \psi_{\mu}\rangle= 1$.

\subsection[Rawnsley and Perelomov coherent states are equivalent for CP\textasciicircum{}n]
{Rawnsley and Perelomov coherent states are equivalent for $\boldsymbol{\mathbb{C}P^n}$}

This follows from \cite{Ra} as mentioned below.
Let $G ={\rm SU}(n+1)$, $K= {\rm S}({\rm U}(n) \times {\rm U}(1)) \subset G$ and~$\psi_0 $ be a non-zero vector in the Hilbert space of geometric quantization such that there exists a~character $\chi\colon K \rightarrow {\mathbb C}^*$ such that $U_k \psi_0 = \chi\big(k^{-1}\big)\psi_0 $. Let~$\mathcal{U}_g$ be an unitary representation of~$G$ on the Hilbert space. Then for~$g\in G$, $e(g)= \mathcal{U}_g \psi_0$ are the states of the Hilbert space which are the global Perelomov states as in \cite[pp.~403--404]{Ra} and they coincide with the Rawnsley coherent states.

\subsection*{Acknowledgement}
The authors would like to thank Gautam Bharali (IISc, Bangalore) and Mahan Mj (TIFR, Mumbai) for the useful discussions on the theory of totally real submanifolds in several complex variables and Proposition~\ref{celldecomp}. They would like to thank the anonymous referees for their valuable suggestions for improvement of the paper. Rukmini Dey acknowledges support from the project RTI4001, Department of Atomic Energy, Government of India and support from grant CRG/2018/002835, Science and Engineering Research Board, Government of India.

\pdfbookmark[1]{References}{ref}
\LastPageEnding


\begin{thebibliography}{99}
\footnotesize\itemsep=0pt

\bibitem{BS}
Berceanu S., Schlichenmaier M., Coherent state embeddings, polar divisors and
 {C}auchy formulas, \href{https://doi.org/10.1016/S0393-0440(99)00075-3}{\textit{J.~Geom. Phys.}} \textbf{34} (2000), 336--358,
 \href{https://arxiv.org/abs/math.DG/9903105}{arXiv:math.DG/9903105}.

\bibitem{Be}
Berezin F.A., Quantization, \href{https://doi.org/10.1070/IM1974v008n05ABEH002140}{\textit{Math. USSR Izv.}} \textbf{8} (1974),
 1109--1165.

\bibitem{Bo}
Boggess A., C{R} manifolds and the tangential {C}auchy--{R}iemann complex,
 \textit{Studies in Advanced Mathematics}, CRC Press, Boca Raton, FL, 1991.

\bibitem{DH}
Doyle P.H., Hocking J.G., A decomposition theorem for {$n$}-dimensional
 manifolds, \href{https://doi.org/10.2307/2034963}{\textit{Proc. Amer. Math. Soc.}} \textbf{13} (1962), 469--471.

\bibitem{Ki}
Kirwin W.D., Coherent states in geometric quantization, \href{https://doi.org/10.1016/j.geomphys.2006.04.007}{\textit{J.~Geom. Phys.}}
 \textbf{57} (2007), 531--548, \href{https://arxiv.org/abs/math.SG/0502026}{arXiv:math.SG/0502026}.

\bibitem{KS}
Klauder J.R., Skagerstam B.S. (Editors), Coherent states: applications in
 physics and mathematical physics, \href{https://doi.org/10.1142/0096}{World Scientific Publishing Co.}, Singapore,
 1985.

\bibitem{K}
Kostant B., Orbits and quantization theory, in Actes du {C}ongr\`es
 {I}nternational des {M}ath\'ematiciens ({N}ice, 1970), {T}ome~2, 1971,
 395--400.

\bibitem{Na}
Nair V.P., Quantum field theory: a modern perspective, \textit{Graduate Texts in
 Contemporary Physics}, \href{https://doi.org/10.1007/b106781}{Springer}, New York, 2005.

\bibitem{N}
Nakahara M., Geometry, topology and physics, 2nd~ed., \textit{Graduate Student Series
 in Physics}, \href{https://doi.org/10.1201/9781420056945}{Institute of Physics}, Bristol, 2003.

\bibitem{Od}
Odzijewicz A., Coherent states and geometric quantization, \href{https://doi.org/10.1007/BF02096666}{\textit{Comm. Math.
 Phys.}} \textbf{150} (1992), 385--413.

\bibitem{Pe}
Perelomov A., Generalized coherent states and their applications, \textit{Texts and
 Monographs in Physics}, \href{https://doi.org/10.1007/978-3-642-61629-7}{Springer-Verlag}, Berlin, 1986.

\bibitem{Rad}
Radcliffe J.M., Some problems of coherent spin states, \href{https://doi.org/10.1088/0305-4470/4/3/009}{\textit{J.~Phys.~A: Gen.
 Phys.}} \textbf{4} (1971), 313--323.

\bibitem{Ra}
Rawnsley J.H., Coherent states and {K}\"ahler manifolds, \href{https://doi.org/10.1093/qmath/28.4.403}{\textit{Quart.~J.
 Math. Oxford}} \textbf{28} (1977), 403--415.

\bibitem{Sch}
Schnabel R., Squeezed states of light and their applications in laser
 interferometers, \href{https://doi.org/10.1016/j.physrep.2017.04.001}{\textit{Phys. Rep.}} \textbf{684} (2017), 1--51,
 \href{https://arxiv.org/abs/1611.03986}{arXiv:1611.03986}.

\bibitem{Sp}
Spera M., On {K}\"ahlerian coherent states, in Geometry, Integrability and
 Quantization ({V}arna, 1999), Coral Press Sci. Publ., Sofia, 2000, 241--256.

\bibitem{Sp2}
Spera M., On some geometric aspects of coherent states, in Coherent states and
 their applications, \textit{Springer Proc. Phys.}, Vol.~205, \href{https://doi.org/10.1007/978-3-319-76732-1_8}{Springer}, Cham,
 2018, 157--172.

\bibitem{W}
Woodhouse N., Geometric quantization, \textit{Oxford Mathematical Monographs}, The
 Clarendon Press, Oxford University Press, New York, 1980.

\bibitem{Ya}
Yaffe L.G., Large {$N$} limits as classical mechanics, \href{https://doi.org/10.1103/RevModPhys.54.407}{\textit{Rev. Modern
 Phys.}} \textbf{54} (1982), 407--435.

\end{thebibliography}
\end{document}